\documentclass[a4paper,12pt]{amsart}

\usepackage{amsfonts, amsmath, amssymb, amstext}

\usepackage{enumitem}
\usepackage{color}

\allowdisplaybreaks[2]

\theoremstyle{plain}
\newtheorem{thm}{Theorem}
\newtheorem{cor}[thm]{Corollary}

\theoremstyle{definition}

\theoremstyle{remark}
\newtheorem{rem}[thm]{Remark}

\def\C{{\mathbb C}}

\def\Z{{\mathbb Z}}

\def\Gg{{\mathcal G}}
\newcommand{\id}{\operatorname{id}}

\usepackage{mathtools}
\usepackage[all]{xy}
\usepackage{tikz}

\begin{document}

\title{All classifiable Kirchberg algebras are $C^{\ast}$-algebras of ample groupoids}

\author{Lisa Orloff Clark}
\author{James Fletcher}
\author{Astrid an Huef}
\address{School of Mathematics and Statistics, Victoria University of Wellington, PO Box 600, Wellington 6140, New Zealand}
\email {lisa.clark@vuw.ac.nz, james.fletcher@vuw.ac.nz, astrid.anhuef@vuw.ac.nz}

\date{\today}
\thanks{This research was supported by the Marsden grant 18-VUW-056 from the
Royal Society of New Zealand.}



\begin{abstract} In this note we show that every Kirchberg algebra in the UCT class is the $C^{\ast}$-algebra of a Hausdorff, ample, amenable and locally contracting groupoid. The non-unital case follows from Spielberg's graph-based models for Kirchberg algebras.  Our contribution is the unital case and our proof builds on Spielberg's construction. \end{abstract}

\maketitle

\section{Introduction}

An ample groupoid is an \'etale groupoid whose topology has a basis of compact open sets.  Its rich topological structure carries over to the associated groupoid $C^{\ast}$-algebra, making this a particularly nice class of algebras to work with.  For example,  having a basis of compact open sets means that the $C^{\ast}$-algebra is the closed span of characteristic functions.

Kirchberg algebras are $C^*$-algebras that are purely infinite, simple, separable and nuclear. The UCT or bootstrap class consists of $C^*$-algebras that are $KK$-equivalent to a separable abelian $C^*$-algebra \cite{Rordam}. The Kirchberg algebras in the UCT class are classified by their $K$-theory by the celebrated classification theorem of Kirchberg  and Phillips \cite{Kirchberg, MR1745197}.  As pointed out by Spielberg in \cite{MR2377136}, because of this classification theorem ``it is possible to prove results about UCT Kirchberg algebras by choosing a convenient model''.

The applications, and hence the models, abound. For example, in \cite{Spielberg-TAMS-2009}, Spielberg shows that
Kirchberg algebras with finitely generated $K$-theory and free $K_1$-group are semiprojective.  To do this he  realises these Kirchberg algebras as  $C^*$-algebras of certain infinite directed graphs, and then shows that these graph $C^*$-algebras are semiprojective. In \cite{RSS}, Ruiz, Sims and S\o rensen  prove that every Kirchberg algebra in the UCT class has nuclear dimension $1$; to do this they show that Kirchberg $2$-graph algebras with trivial $K_0$-group and finite $K_1$-group have nuclear dimension $1$, and that every Kirchberg algebra in the UCT class is a direct limit of $2$-graph algebras. In \cite{MR2377136}, Spielberg shows that all prime-order
automorphisms of $K$-groups of Kirchberg algebras are induced by automorphisms of Kirchberg
algebras having the same order; to do this he uses his graph-based models for Kirchberg algebras from \cite{MR2329002}.

Spielberg uses his construction in \cite{MR2329002} to realise all non-unital Kirchberg algebras as $C^*$-algebras of ample groupoids. In this note we trace through Spielberg's work, and show how it can be adapted to the unital case where we have to keep track of the class of the identity in the $K_0$-group.  Katsura proved  that all Kirchberg algebras in the UCT class are isomorphic to $C^*$-algebras of topological graphs  \cite{Katsura2008}, which are then $C^*$-algebras of amenable, \'etale groupoids by \cite{Yeend}. But the groupoids of topological graphs are usually not ample, and so our model of all Kirchberg algebras as $C^*$-algebras of ample groupoids is new.

There are other models of Kirchberg algebras of interest, and many are based on  groupoids. For example, R\o rdam and Sierakowski show that a countable discrete group admits an action on the Cantor set such that the crossed product is a Kirchberg algebra in the UCT class if and only if $G$ is exact and non-amenable \cite{Rordam-Sierakowski}.  Brown et al developed a technique for realising many Kirchberg algebras as the $C^*$-algebras of amenable, principal groupoids \cite{BCSS}.

This modelling strategy is of course not restricted to Kirchberg algebras and can be used whenever an appropriate classification theorem exists. For example, Putnam models simple tracially AF algebras as $C^*$-algebras of minimal, amenable, \'etale equivalence relations \cite{Putnam}, using the classification theorem of Lin from \cite{Lin}.

\section{Unital Kirchberg algebras}

Let  $G_0$ and $G_1$ be countable abelian groups and $c\in G_0$. By \cite[Theorem~5.6]{Eliott-Rordam},  there exists a  unital Kirchberg algebra $A$ in the UCT class such that   $(K_0(A),  [1_A], K_1(A))$ is isomorphic to $(G_0,  c, G_1)$.
Further, unital Kirchberg algebras in the UCT class are classified by their $K$-groups and  the class of the identity in  $K_0$  \cite[Theorem~4.2.4]{MR1745197}. Thus if $A$ and $B$ are unital UCT Kirchberg algebras and there exists a graded isomorphism $\psi:K_*(A)\rightarrow K_*(B)$ such that $\psi([1_A])=[1_B]$, then there exists an isomorphism $\pi:A\rightarrow B$ such that $\pi_*=\psi$.

\begin{thm}\label{thm-main}
Let $A$ be a unital Kirchberg $C^{\ast}$-algebra in the UCT class. Then there exists a Hausdorff, ample, amenable and locally contracting  groupoid $\Gg$ such that $A$ is isomorphic to the $C^{\ast}$-algebra of $\Gg$.
\end{thm}

\begin{proof}
Let $G_0$ and $G_1$ be countable abelian groups and let $c\in G_0$.  It suffices to show that  there exists a groupoid with the stated properties such that its $C^{\ast}$-algebra is a unital Kirchberg algebra in the UCT class with $K$-groups isomorphic to  $(G_0, G_1)$, and that this isomorphism carries the class of the identity  to $c$.

We start by following the steps briefly outlined on page~367 of \cite{MR2329002} which prove that there exists a Hausdorff, ample, amenable and locally contracting groupoid $\Gg$ such that $C^{\ast}(\Gg)$ is a non-unital Kirchberg algebra and $K_*(C^{\ast}(\Gg))=(G_0, G_1)$.  We will then show that there is a reduction $\Gg'$ of $\Gg$ that meets our requirements.

We apply \cite[Theorem~2.1]{MR2377136}  to get directed graphs $E_0$ and $E_1$ such that
\begin{itemize}
\item each $E_i$ is strongly connected\footnote{Spielberg calls strongly connected graphs irreducible. The connection is that  the graph is strongly connected if and only if the vertex matrix is irreducible.},
\item each element of $G_i$ is a vertex in $E_i$,
\item
$K_*(C^{\ast}(E_0))\cong (G_0, 0) \text{ and }K_*(C^{\ast}(E_1))\cong (G_1, 0)$, and
\item for $g\in G_i$, the   isomorphism of $K_0(C^{\ast}(E_i))$ onto $G_i$ sends  the class  of the vertex projection $p_g$ to $g$ in $G_i$.
\end{itemize}
(In the notation of \cite[Theorem~2.1]{MR2377136}, we apply the theorem with $G=G_i$, $A=G_i$, $\Gamma=\{\id\}$ and $\pi_0=\id$.)

We let $F_0$ be the directed graph with one vertex $w$ and infinitely many loops; then  $K_*(C^{\ast}(F_0))\cong (\Z, 0)$.  We let $F_1$ be any strongly connected directed graph such that $K_*(C^{\ast}(F_1))\cong (0,\Z)$. (For example, depending on conventions for paths, take $F_1$ to be the graph of \cite[Figure~6]{MR2377136} or its opposite graph.)

Since $K_*(C^{\ast}(F_0))\cong (\Z, 0)$ is torsion free, the  K\"unneth Theorem for tensor products (\cite[Theorem~2.14]{Schochet}) gives a graded isomorphism
\begin{align*}
\alpha: K_0\big(C^{\ast}(E_0)&\otimes C^{\ast}(F_0)\big)\oplus K_1\big(C^{\ast}(E_0)\otimes C^{\ast}(F_0)  \big)
\\
&\to \big(K_0(C^{\ast}(E_0))\oplus K_1(C^{\ast}(E_0)) \big)\otimes\big( K_0(C^{\ast}(F_0))\oplus K_1(C^{\ast}(F_0))\big).\notag
\end{align*}
Taking $0$-graded parts gives an isomorphism
\begin{align*}\label{iso-Kunneth2}
\alpha:  K_0\big( C^{\ast}(E_0)&\otimes  C^{\ast}(F_0)\big)\\
&\to  \big( K_0(C^{\ast}(E_0))\otimes K_0(C^{\ast}(F_0))\big)\oplus \big(K_1(C^{\ast}(E_0))\otimes K_1(C^{\ast}(F_0))\big),\notag
\end{align*}
and this gives
\begin{equation*}\label{eq-alpha}
K_0\big( C^{\ast}(E_0)\otimes  C^{\ast}(F_0)\big)
\stackrel{\alpha}{\cong} (G_0\otimes\Z) \oplus (0\otimes 0)\cong G_0.
\end{equation*}
The Cartesian product  $E_0\times F_0$ is a $2$-graph, and, adapting the argument of \cite[Corollary~3.5(iv)]{MR1745529} to the more general finitely aligned setting,
we have
\begin{equation*}\label{eq-2graph}
C^{\ast}(E_0\times F_0)\cong C^{\ast}(E_0)\otimes C^{\ast}(F_0);
\end{equation*}
the isomorphism sends the vertex projection $p_{(c,w)}$ to $p_c\otimes p_w$.
Thus
\begin{equation}\label{eq-00} K_0(C^{\ast}(E_0\times F_0))\cong G_0.
\end{equation}
Similar calculations show that
\begin{equation}\label{more-Kunneth}
K_1(C^{\ast}(E_0\times F_0))=0, \quad K_0(C^{\ast}(E_1\times F_1)) =0 \text{\ and\ } K_1(C^{\ast}(E_1\times F_1))\cong  G_1.
\end{equation}

Taking the union of $E_0\times F_0$ and $E_1\times F_1$ gives another $2$-graph. Since $E_0$ and $E_1$ are not row-finite, neither is  $(E_0\times F_0)\sqcup(E_1\times F_1)$, but it is a finitely aligned $2$-graph. So there is an associated boundary-path groupoid whose $C^{\ast}$-algebra  is isomorphic to $C^{\ast}((E_0\times F_0)\sqcup(E_1\times F_1))$  by \cite[Theorem~6.13]{MR2184052}. Now  \[C^{\ast}((E_0\times F_0)\sqcup(E_1\times F_1))\cong C^{\ast}(E_0\times F_0)\oplus C^{\ast}(E_1\times F_1)\]   has $K$-theory $(G_0, 0)\oplus (0,G_1)\cong(G_0, G_1)$, but it is not simple, and so is not a Kirchberg algebra.

The construction in  \cite{MR2329002} of a hybrid graph $\Omega$ is   a way of gluing $E_0\times F_0$ and $E_1\times F_1$ together in such a way that the associated $C^{\ast}$-algebra is a  non-unital Kirchberg algebra in the UCT class, but retains the same $K$-theory as $C^{\ast}((E_0\times F_0)\sqcup(E_1\times F_1))$.  It is non-unital because the graph $F_1$ has infinitely many vertices.
We now go through the key points of the construction from  \cite{MR2329002} that we need.

In \cite[Definition~2.2]{MR2329002} the hybrid graph $\Omega$ is defined.  Loosely speaking, this hybrid graph is formed by joining together the $2$-graphs $E_0\times F_0$ and $E_1\times F_1$ with a directed graph in the middle and imposing no additional factorisation rules.
Using the collection of infinite paths in $\Omega$, a second-countable locally compact Hausdorff ample groupoid $\mathcal{G}$ is constructed (see \cite[Definitions~2.15 and~2.17, and Lemma~2.14]{MR2329002}). This groupoid is topologically principal, minimal and locally contracting  by \cite[Lemma~2.18]{MR2329002}.  Since $\Gg$  is  topologically principal and minimal, its reduced $C^{\ast}$-algebra is simple by \cite[Proposition~4.6]{Renault}, and since $\Gg$ is topologically principal and locally contracting, its reduced $C^{\ast}$-algebra is purely infinite by \cite[Proposition~2.4]{AD}.

Spielberg also considers the universal  $C^{\ast}$-algebra $C^{\ast}(\Omega)$ of the hybrid graph; this $C^{\ast}$-algebra is generated by projections and partial isometries associated to  vertices and edges of $\Omega$, respectively,   subject to relations given in \cite[Definition~3.3]{MR2329002}. These relations resemble the Cuntz--Krieger relations for graph and higher-rank graph algebras.
In \cite[Corollary~3.19]{MR2329002} a type of gauge-invariant uniqueness theorem is used to show that there is an isomorphism
\begin{equation}\label{eq-pi}
\pi:C^{\ast}(\Omega)\to C^{\ast}(\Gg)
\end{equation} and that $C^{\ast}(\mathcal{G})\cong C_r^*(\mathcal{G})$.  Further, $C^{\ast}(\Gg)$ is shown to be Morita equivalent, hence stably isomorphic,  to the crossed product of an AF algebra by $\Z^2$. It follows that  $C^{\ast}(\mathcal{G})$ is  nuclear \cite[Theorem~15]{MR2276659} and is in the UCT class \cite[Proposition~2.4.7]{Rordam}. Since $\Gg$ is \'etale, the nuclearity of $C^{\ast}_r(\Gg)$ implies that $\Gg$ is measurewise amenable \cite[Corollary~6.2.14]{ADR}, and since $\Gg$ has countable orbits it then follows from  \cite[Theorem~3.3.7]{ADR} that $\Gg$ is amenable.

Using that $C^{\ast}(\Gg)$ is universal for  generators and relations, it is proved in  \cite[Theorem~4.7]{MR2329002} that the inclusion of $(E_0\times F_0)\sqcup (E_1\times F_1)$ in the hybrid graph $\Omega$ induces an injective  homomorphism 
\begin{equation}
i:C^{\ast}((E_0\times F_0)\sqcup (E_1\times F_1))\to C^{\ast}(\Omega),
\end{equation}
 and that this homomorphism induces an isomorphism at the level of $K$-theory.  Thus $K_*(C^{\ast}(\mathcal{G}))\cong (G_0,G_1)$). As observed by Spielberg, this shows that every non-unital Kirchberg algebra in the UCT class is the $C^{\ast}$-algebra of an ample, amenable and locally contracting groupoid.

We now restrict the groupoid $\Gg$ to prove our theorem. The isomorphism $\pi:C^{\ast}(\Omega)\to C^{\ast}(\mathcal{G})$ mentioned at \eqref{eq-pi}  is defined on page~359 of \cite{MR2329002}. It carries the vertex projection $p_{(c,w)}\in C^{\ast}(\Omega)$ corresponding to the vertex $(c,w)\in E_0\times F_0\subset\Omega$ to the characteristic function $1_U\in C^{\ast}(\Gg)$, where  $U:=Z((c,w))\times\{0\}\times Z((c,w))$ is a compact open subspace of the unit space of $\Gg$ (see the  first sentence of \cite[\S3]{MR2329002} and \cite[Definition~2.17]{MR2329002}).
Set \[\Gg'\coloneqq U\Gg U=\{g\in \Gg\colon s(g), r(g)\in U\}.\]
Then $\Gg'$ is a locally compact groupoid with compact unit space $U$; it is minimal, topologically principal, locally contracting and amenable because $\Gg$ is. The amenability of $\Gg'$ implies that $C^{\ast}(\Gg')$ is in the UCT class  \cite[Theorem~0.1]{Tu}.

For $f\in C_c(\Gg)$ we have $1_Uf1_U\neq 0$ if and only  if $f$ has support in $U\Gg U$. Since $U\Gg U$ is open in $\Gg$, the inclusion induces a homomorphism $\iota: C_c(\Gg') \to C_c(\Gg)$ which has range $1_U C_c(\Gg) 1_U$. Since $\iota$ is bounded, it extends to a homomorphism $\iota: C^{\ast}(\Gg) \to C^{\ast}(\Gg')$ with range $1_U C^{\ast}(\Gg) 1_U$. Similarly, restriction of functions $r:C_c(\Gg)\to C_c(\Gg')$ is bounded and extends to a homomorphism $r:C^{\ast}(\Gg)\to C^{\ast}(\Gg')$.  It is easy to check that $\iota\circ r(f)=f$ and $r\circ \iota(g)=g$ for $f\in C_c(\Gg)$ and $g\in 1_U C_c(\Gg) 1_U$. Thus
\begin{equation}\label{eq-iota}
\iota:C^{\ast}(\Gg')\to 1_U C^{\ast}(\Gg) 1_U
\end{equation}
is an isomorphism.
In summary, we have the following $C^{\ast}$-algebras and injective homomorphisms between them:

\begin{equation*}\label{eq-injections}
\begin{tikzpicture}[>=stealth, yscale=0.6]
\node (P1) at (-2,0) {$C^{\ast}((E_0\times F_0)\sqcup (E_1\times F_1))$};
\node (P2) at (3,0) {$C^{\ast}(\Omega)$};
\node (P3) at (7,0) {$C^{\ast}(\Gg)$};
\node (P4) at (7,-3) {$1_UC^{\ast}(\Gg)1_U$};
\node (P5) at (3,-3) {$C^{\ast}(\Gg')$};
\draw[->](P1)--(P2) node[above, midway] {$i$};
\draw[->] (P2)--(P3) node[above, midway] {$\pi$} node[below, midway] {$\cong$};
\draw[->] (P4)--(P3) node[right, midway] {$j$};
\draw[->] (P5)--(P4)node[below, midway] {$\cong$} node[above, midway] {$\iota$};
\end{tikzpicture}
\end{equation*}
where $\pi, i,\iota$ were discussed at \eqref{eq-pi}--\eqref{eq-iota} and $j$ is inclusion.
 Since $C^{\ast}(\Gg)$ is simple, it follows that $\overline{C^{\ast}(G)1_UC^{\ast}(\Gg)}=C^{\ast}(\Gg)$. Thus $1_UC^{\ast}(\Gg)1_U$  is a full corner in $C^{\ast}(\Gg)$ and $j$ induces an isomorphism on $K$-theory by \cite[Proposition~1.2]{Paschke}. Thus we obtain an isomorphism
 \begin{equation}\label{eq-composition}
( i_*)^{-1}\circ(\pi^{-1}\circ j\circ \iota)_*: K_*(C^{\ast}(\Gg'))\to K_*\big(C^{\ast}((E_0\times F_0)\sqcup (E_1\times F_1)) \big).
 \end{equation}
Combining~\eqref{eq-composition} with  the calculations at \eqref{eq-00}--\eqref{more-Kunneth} give that $K_* (C^{\ast}(\Gg') ) \cong (G_0, G_1)$.

It remains to show that the isomorphism of $K_0 (C^{\ast}(\Gg') )$ onto $G_0$ sends $[1_U]$ to $c$. We have
\[
( i_*)^{-1}\circ(\pi^{-1}\circ j\circ \iota)_*([1_U])=( i_*)^{-1}([\pi^{-1}(1_U)])=( i_*)^{-1}([p_{(c,w)}])=[p_{(c,w)}].
\]
Thus it remains  to trace through
\begin{align*}
K_0\big(C^{\ast}((E_0\times F_0)\sqcup &(E_1\times F_1)) \big)\cong K_0\big(C^{\ast}(E_0\times F_0) \big)\cong K_0\big( C^{\ast}(E_0)\otimes C^{\ast}(F_0)\big)\\
&\stackrel{\alpha^{-1}}{\cong\ }K_0(C^{\ast}(E_0))\otimes K_0(C^{\ast}(F_0))\cong K_0(C^{\ast}(E_0))\cong G_0.
\end{align*}
It follows from Remark~\ref{remark-kunneth} below  that $\alpha([p_c]\otimes [1_{C^{\ast}(F_0)}])=[p_c\otimes 1_{C^{\ast}(F_0)}]$, and thus we have
\[
[p_{(c,w)}] \mapsto [p_{(c,w)}]\mapsto [p_c\otimes p_w]=[p_c\otimes 1_{C^{\ast}(F_0)}]\mapsto [p_c]\otimes [1_{C^{\ast}(F_0)}]\mapsto [p_c]\mapsto c.\qedhere
\]
\end{proof}

\begin{rem}\label{remark-kunneth} Let $A$ and $B$ be unital $C^{\ast}$-algebras with $K_*(B)$ torsion free.  Then the    K\"unneth isomorphism $\alpha:K_0(A)\otimes K_0(B)\to K_0(A\otimes B)$ of \cite[Theorem~2.14]{Schochet}
 is induced by the natural map from $M_r(A)\otimes M_s(B)$ to $M_{rs}(A\otimes B)=M_r(M_s(A\otimes B))$ such that
 \[a\otimes b=(a_{ij})_{i,j\in\{1,\ldots,r\}}\otimes (b_{kl})_{k,l\in\{1,\ldots,s\}}\mapsto ((a_{ij}\otimes b_{kl})_{k,l\in\{1,\ldots,s\}})_{i,j\in\{1,\ldots,r\}}.\] Then $\alpha$  is extended to non-unital $A$ using the exact sequence $0\to A\to A^+\to \C\to 0$.  It follows that when $A$ is non-unital, we have  $\alpha([p]\otimes[1_B])=[p\otimes 1_B]$ for a projection $p$ in $A$.
\end{rem}

Combining Theorem~\ref{thm-main} with Spielberg's result for non-unital Kirchberg algebras (which we outlined in the proof above) we obtain:

\begin{cor} All Kirchberg algebras in the UCT class are isomorphic to $C^*$-algebras of Hausdorff, ample, amenable and locally contracting  groupoids.
\end{cor}


\begin{thebibliography}{10}

\bibitem{AD} Claire Anantharaman-Delaroche.  Purely infinite $C^{\ast}$-algebras arising from dynamical systems. \emph{Bull. Soc. Math. France}, {125}:199--225, 1997.

\bibitem{ADR}
 Claire Anantharaman-Delaroche and Jean Renault. \emph{Amenable groupoids}, volume~36 of  \emph{Monographies de L'Enseignement Math\'ematique},  L'Enseignement Math\'ematique, Geneva, 2000.


%
%

\bibitem{BCSS}
Jonathan H. Brown, Lisa Orloff Clark, Adam Sierakowski, and Aidan Sims.
Purely infinite simple $C^{\ast}$-algebras that are principal groupoid $C^{\ast}$-algebras. \emph{J. Math. Anal. Appl.}, 439(1): 213--234, 2016.




\bibitem{Eliott-Rordam}  George A. Elliott and Mikael R\o rdam.  Classification of certain infinite simple $C^{\ast}$-algebras. II. \emph{Comment. Math. Helv.}, 70: 615--638, 1995.


\bibitem{MR2184052}
Cynthia Farthing, Paul~S. Muhly, and Trent Yeend.
\newblock Higher-rank graph {$C^{\ast}$}-algebras: an inverse semigroup and groupoid
  approach.
\newblock {\em Semigroup Forum}, 71(2):159--187, 2005.

\bibitem{MR2276659}
Astrid an~Huef, Iain Raeburn, and Dana~P. Williams.
\newblock Properties preserved under {M}orita equivalence of {$C^{\ast}$}-algebras.
\newblock {\em Proc. Amer. Math. Soc.}, 135(5):1495--1503, 2007.


\bibitem{Katsura2008}
Takeshi Katsura.
A class of $C^{\ast}$-algebras generalizing both graph algebras and homeomorphism $C^{\ast}$-algebras IV, pure infiniteness.
\emph{J. Funct. Anal.}, 254(5): 1161--1187, 2008.

\bibitem{Kirchberg} Eberhard Kirchberg. The classification of purely infinite $C^*$-algebras using Kasparov's theory. Preprint, 1994.

\bibitem{MR1745529}
Alex Kumjian and David Pask.
\newblock Higher rank graph {$C^\ast$}-algebras.
\newblock {\em New York J. Math.}, 6:1--20, 2000.

\bibitem{Lin}
Hiaxin Lin.  Classification of simple $C^*$-algebras of tracial topological rank zero. \emph{Duke Math. J.}, 125:91--119, 2004.

\bibitem{Paschke} William L. Paschke. $K$-theory for  actions of the circle group on $C^{\ast}$-algebras.
\emph{J. Operator Theory}, 6:125--133, 1981.

\bibitem{MR1745197}
N.~Christopher Phillips.
\newblock A classification theorem for nuclear purely infinite simple
  {$C^{\ast}$}-algebras.
\newblock {\em Doc. Math.}, 5:49--114, 2000.


\bibitem{Putnam}
Ian F. Putnam.
Some classifiable groupoid $C^{\ast}$-algebras with prescribed $K$-theory. \emph{Math. Ann.}, 370(3,4): 1361--1387, 2018.

%

\bibitem{Renault} J. Renault, \emph{A groupoid approach to $C^{\ast}$-algebras},  volume~793 of \emph{Lecture Notes in Mathematics},  Springer, Berlin, 1980.




\bibitem{Rordam-Sierakowski} Mikael R\o rdam and Adam Sierakowski.
Purely infinite $C^*$-algebras arising from crossed products.
\emph{Ergodic Theory Dynam. Systems}, 32(1):273--293, 2012.

\bibitem{Rordam}  Mikael R\o rdam and Erling St\o rmer. \emph{Classification of nuclear $C^{\ast}$-algebras. Entropy in operator algebras}, volume~126 of  \emph{Encyclopaedia of Mathematical Sciences},  Springer-Verlag, Berlin, 2002.

\bibitem{RSS} Efren Ruiz,  Aidan Sims, and Adam P. W. S\o rensen. UCT-Kirchberg algebras have nuclear dimension one. \emph{Adv. Math.}, 279:1--28,  2015.


\bibitem{Schochet} Claude Schochet. Topological methods for $C^{\ast}$-algebras. II. Geometric resolutions and the K\"unneth formula. \emph{Pacific J. Math.},  98:443--458,  1982.

\bibitem{MR2329002}
Jack Spielberg.
\newblock Graph-based models for {K}irchberg algebras.
\newblock {\em J. Operator Theory}, 57(2):347--374, 2007.

\bibitem{Spielberg-TAMS-2009}
Jack Spielberg.
Semiprojectivity for certain purely infinite $C^*$-algebras.
\emph{Trans. Amer. Math. Soc.}, 361(6): 2805--2830, 2009.

\bibitem{MR2377136}
Jack Spielberg.
\newblock Non-cyclotomic presentations of modules and prime-order automorphisms
  of {K}irchberg algebras.
\newblock {\em J. Reine Angew. Math.}, 613:211--230, 2007.

\bibitem{Tu} J.L. Tu.  La conjecture de Baum-Connes pour les feuilletages moyennables.  \emph{K-Theory},  17:215--264,  1999.

\bibitem{Yeend} Trent Yeend. Groupoid models for the $C^{\ast}$-algebras of topological higher-rank graphs. \emph{J. Operator Theory}, 57(1): 95--120, 2007.

\end{thebibliography}
\end{document}